\documentclass[11pt]{article}
\usepackage[english]{babel}
\usepackage{amsmath,amsthm}
\usepackage[T1]{fontenc}
\usepackage{amsfonts}
\usepackage{mathrsfs}
\usepackage{color}
\usepackage{bbm}
\usepackage{enumerate}
\usepackage{amsrefs}
\usepackage{mathtools}
\usepackage{hyperref}

\newtheorem{theorem}{Theorem}[section]

\newtheorem{lemma}[theorem]{Lemma}

\newtheorem{assumption}{Assumption}
\theoremstyle{definition}
\newtheorem{definition}[theorem]{Definition}
\theoremstyle{remark}
\newtheorem{remark}[theorem]{Remark}
\newtheorem{example}[theorem]{Example}
\numberwithin{equation}{section}
\begin{document}
\def\Pro{{\mathbb{P}}}
\def\E{{\mathbb{E}}}
\def\e{{\varepsilon}}
\def\veps{{\varepsilon}}
\def\ds{{\displaystyle}}
\def\nat{{\mathbb{N}}}
\def\Dom{{\textnormal{Dom}}}
\def\dist{{\textnormal{dist}}}
\def\R{{\mathbb{R}}}
\def\O{{\mathcal{O}}}
\def\T{{\mathcal{T}}}
\def\Tr{{\textnormal{Tr}}}
\def\sgn{{\textnormal{sign}}}
\def\I{{\mathcal{I}}}
\def\A{{\mathcal{A}}}
\def\H{{\mathcal{H}}}
\def\S{{\mathcal{S}}}

\title{Global solutions to the stochastic reaction-diffusion equation with superlinear accretive reaction term and superlinear multiplicative noise term on a bounded spatial domain}%
\author{M. Salins\\ Boston University \\ msalins@bu.edu}
\maketitle

\begin{abstract}
We describe sufficient conditions on the reaction terms and multiplicative noise terms of a stochastic reaction-diffusion equation that guarantee that the solutions never explode. Both the reaction term and multiplicative noise terms are allowed to grow superlinearly.
\end{abstract}

%
%

\section{Introduction}
We describe sufficient conditions on the reaction terms $f:\mathbb{R} \to \mathbb{R}$ and multiplicative noise terms $\sigma: \mathbb{R} \to \mathbb{R}$ that guarantee that the mild solutions to the stochastic reaction-diffusion never explode. The stochastic reaction-diffusion equation is defined on an open bounded spatial domain $D \subset \mathbb{R}^d$ with sufficiently smooth boundary
\begin{equation} \label{eq:SPDE}
  \begin{cases}
    \frac{\partial u}{\partial t} (t,x) = \mathcal{A} u(t,x)  + f(u(t,x)) + \sigma(u(t,x)) \dot{w}(t,x)\\
    u(t,x) = 0 , \ \ \ \ x \in \partial D\\
    u(0,x)  = u_0(x).
  \end{cases}
\end{equation}
In the above equation, $\mathcal{A}$ is a second-order elliptic operator, and $\dot{w}$ is a sufficiently regular stochastic noise that is white in time and generally correlated in space.

The global well-posedness of the mild solution to \eqref{eq:SPDE} is well-known when $f$ and $\sigma$ are both globally Lipschitz continuous with at most linear growth \cite{walsh, dpz,d-1999, k-1992}. Many authors have investigated the well-posedness of such an equation when $f$ is not Lipschitz continuous, but features some monotonicity condition \cite{c-2003,pr,gp-1993,mr-2010} or ``half-Lipschitz'' condition where $f$ can be written as the sum of a decreasing function and a line. For example, when $f$ is an odd-degree polynomial with negative leading term such as $f(u) = -u^3 + u$, then $\limsup_{|u| \to +\infty} \frac{f(u)}{u} = -\infty$. Such a strongly dissipative force $f$ pushes solutions toward the finite part of the space, which prevents explosion if $\sigma$ is not too strong.

In this paper, we are most interested in the case of superlinear accretive forcing
\begin{equation}
  \liminf_{|u| \to \infty} \frac{f(u)}{u} = +\infty
\end{equation}
and superlinear multiplicative noise
\begin{equation}
  \liminf_{|u| \to \infty} \frac{|\sigma(u)|}{|u|} = +\infty.
\end{equation}
Unlike in the cases of dissipative forcing terms, a superlinear accretive force pushes solutions towards $\pm \infty$. The central question is whether these forces cause the solutions to explode in finite time, or whether the solution is global in time.

For a deterministic one-dimensional ordinary differential equation (ODE) with a positive force $f$,
\begin{equation} \label{eq:intro-ODE}
  \frac{d}{dt} v(t)  = f(v(t)), \ \ \ v(0) = c>0
\end{equation}
the so-called ``Osgood condition'' characterizes whether solutions explode  in finite time. Solutions to this ODE explode in finite time if and only if
\begin{equation}
  \int_c^\infty \frac{1}{f(v)}dv < +\infty.
\end{equation}

Recently, many researchers have investigated finite-time explosion for SPDEs. Bonder and Groisman proved that in the case of a one-dimensional interval spatial domain and space-time white noise, whenever $\int_c^\infty 1/f(u)du<+\infty$ for some $c>0$ and $\sigma$ is a positive constant, solutions to \eqref{eq:SPDE} explode in finite time \cite{bg-2009}. Foondun and Nualart generalized Bonder and Groisman's result to higher-dimensional spatial domains and different kinds of stochastic forcing, and demonstrated that, along with a few additional technical conditions,  the Osgood condition fully characterized explosion \cite{fn-2021} . Solutions never explode if $\int_c^\infty 1/f(u)=+\infty$ for all $c>0$ and solutions explode in finite time if $\int_c^\infty 1/f(u)du < +\infty$ for some $c>0$. Foondun and Nualart's investigation considered only the case of bounded multiplicative noise term $\sigma$.

Previously, Mueller and collaborators investigated explosion properties of the one-dimensional stochastic heat equation when $\mathcal{A}= \frac{\partial^2}{\partial x^2}$, $\dot{w}$ is space-time white noise, the forcing term $f\equiv 0$, and $\sigma$ grows polynomially $\sigma(u) \approx |u|^\gamma$ \cite{m-1991,ms-1993,m-1997,m-2000}. In this case, the dissipativity of the differential operator $\frac{\partial^2}{\partial x^2}$ counteracts some of the expansive properties of the multiplicative stochastic perturbations, and the solutions are global in time if $\gamma< \frac{3}{2}$ and the solutions explode in finite time if $\gamma>\frac{3}{2}$ \cite{m-2000}. Recent investigations proved that if $f$ is strongly dissipative, such as $f(u) = -|u|^{m-1}u$ for large $m$, then $\sigma$ can grow even faster than $|u|^{\frac{3}{2}}$ and solutions will not explode \cite{s-2021}.

Dalang, Khoshnevian, and Zhang proved that solutions  to \eqref{eq:SPDE} can be global in time in some cases when the reaction term $f$ is accretive and superlinear and the multiplicative $\sigma$ term is also superlinear \cite{dkz-2019}. They considered the case of a stochastic heat equation defined on a one-dimensional spatial domain and perturbed by space-time white noise. In this setting they showed that if $|f(u)| \leq C(1 + |u|\log|u|)$ and $\sigma \in o( |u|(\log|u|)^{1/4})$, then solutions never explode. They used two different approaches in their paper: a log-Sobolev approach, and what they call the $L^\infty$ approach. Unfortunately, the $|u|\log|u|$ growth restriction is much more restrictive than the Osgood condition identified by \cite{bg-2009,fn-2021}.

Similar explosion properties have been investigated for the stochastic wave equation with superlinear forcing \cite{ms-2021,fn-2020}.

The main result of this paper is an extension of results from both {\cite{fn-2021}} and \cite{dkz-2019}. We prove that the Osgood condition on the reaction term, identified by {\cite{fn-2021}} in the case of bounded multiplicative noise coefficients $\sigma$, also guarantees that solutions never explode when the multiplicative noise terms are superlinear and satisfy a restriction similar to the one identified by \cite{dkz-2019}.

Assumption \ref{assum:f-sigma}, below, requires that $f$ and $\sigma$ are locally Lipschitz continuous and that there exists an increasing $h: (0,+\infty) \to (0,+\infty)$ such that
 \begin{equation} \label{eq:intro-osgood}
   \int_c^\infty \frac{1}{h(u)}du = +\infty \text{ for some } c>0
 \end{equation}
 and a constant $\gamma \in \left(0, \frac{1-\eta}{2} \right)$ such that
\begin{equation} \label{eq:intro-bound}
  |f(u)| \leq h(|u|)   \ \ \ \ \text{ and } \ \ \ \ |\sigma(u)| \leq |u|^{1-\gamma}(h(|u|))^\gamma \ \ \ \text{ for } |u|>1.
\end{equation}
The number $\eta \in [0,1)$ is a constant identified by Cerrai \cite{c-2009} and defined below in Assumption \ref{assum:noise} that describes a balance between the eigenvalues of $\mathcal{A}$ and the spatial regularity of the noise.  When $\eta \in [0,1)$, solutions to the stochastic reaction-diffusion equation are continuous  in space and  time. Trace-class noise implies that $\eta =0$ and in the case of the one-dimensional heat equation with space-time white noise $\eta> \frac{1}{2}$. This means that in the space-time white noise setting $\gamma$ can be any number in $(0,1/4)$. When $h(u)=u\log u$, \eqref{eq:intro-bound} says that  $\sigma(u)$ can grow as fast as $|u|(\log|u|)^{\gamma}$ for any $\gamma \in (0,1/4)$, which is very similar to, but not the same as, the $o(|u| (\log|u|)^{1/4})$  constraint identified in \cite{dkz-2019}. The result in this article also holds in arbitrary spatial dimension and with different kinds of stochastic forcing.

This Osgood-type condition \eqref{eq:intro-osgood}--\eqref{eq:intro-bound} allows us to consider forces $f$ that grow much faster than $|u|\log|u|$. For example,
\begin{equation}
  f(u) =   u\log(e^e + |u|)\log\log(e^e + |u|)
\end{equation}
satisfies the assumptions \eqref{eq:intro-osgood}--\eqref{eq:intro-bound}, but grows too fast to be covered by the assumptions of \cite{dkz-2019}. {In the particular setting of a one-dimensional heat equation exposed to space-time white noise,} our restriction on the superlinear growth rate of the multiplicative noise term $\sigma$ is neither fully weaker or stronger than the condition identified in \cite{dkz-2019}. There are examples of $\sigma$ that satisfy the assumptions of \cite{dkz-2019}, but fail to satisfy the assumptions of the current paper and there are also $\sigma$ that satisfy the assumptions of the current paper but fail to satisfy the assumptions of \cite{dkz-2019}. Furthermore, it is doubtful that the methods used in this paper can be used to prove the results for every case covered by the assumptions \cite{dkz-2019}, suggesting that both methods are needed to understand these explosion phenomena. A full discussion of these examples is in Section  \ref{S:examples}.

The proof of the global existence and uniqueness of solutions to the reaction-diffusion equation in this paper is quite different than all of the methods used in \cite{dkz-2019} and \cite{fn-2020}, but similar in some ways to the arguments used by Mueller \cite{m-1991}.

To motivate the method used to prove global existence of mild solutions to \eqref{eq:SPDE}, consider  the following non-standard proof of global existence for the ODE \eqref{eq:intro-ODE} under the additional assumption that $f$ satisfies \eqref{eq:intro-osgood}--\eqref{eq:intro-bound}. Let $v(t)$ denote the solution to the ODE \eqref{eq:intro-ODE} and define the times $t_n := \inf\{t>0: v(t) = 2^n\}$. For large $n$, $v$ will satisfy the expression
\begin{align*}
  2^n &= v(t_n)
  = v(t_{n-1}) + \int_{t_{n-1}}^{t_n} f(v(s))ds.
\end{align*}
Because $v(t_{n-1}) = 2^{n-1}$ and $f(v(s)) \leq h(|v(s)|) \leq h(2^n)$ for $s \in [t_{n-1},t_{n}]$, it holds that
\begin{equation*}
  2^n \leq 2^{n-1} + (t_n - t_{n-1})h(2^n),
\end{equation*}
which simplifies to
\begin{equation*}
  \frac{2^{n-1}}{h(2^n)} \leq t_n -t_{n-1}.
\end{equation*}
This expression is enough to imply that the solution to the ODE \eqref{eq:intro-ODE} cannot explode in finite time because for $N$ large enough that $2^{N-1}>v(0)$,
\[\lim_{n \to \infty} t_n \geq \sum_{n=N}^\infty (t_n - t_{n-1}) \geq \sum_{n=N}^\infty \frac{2^{n-1}}{h(2^n)}=+\infty.\]
The above sum is divergent because $h$ is increasing and $2^{n-1} = 2^n - 2^{n-1}$, implying that
\[\sum_{n=N}^\infty\frac{2^{n-1}}{h(2^n)} \geq \sum_{n=N}^\infty\frac{2^n - 2^{n-1}}{h(2^n)} \geq \sum_{n=N}^\infty \int_{2^{n-1}}^{2^n} \frac{1}{h(x)}dx \geq \int_{2^{N-1}}^\infty \frac{1}{h(x)}dx\]
which is $+\infty$ by assumption \eqref{eq:intro-osgood}.

We use a similar method to prove global existence of mild solutions to the SPDE \eqref{eq:SPDE}. The mild solution of \eqref{eq:SPDE} is the solution of the integral equation
  \begin{align} \label{eq:intro-mild}
    u(t) = &S(t) u(0) + \int_{0}^{t} S(t -s) f(u( s))ds + \int_{0}^{t } S( t-s) \sigma(u(s))dw(s) \nonumber\\
  \end{align}
  where $S(t)$ is the semigroup generated by $\mathcal{A}$ and where the spatial variable has been suppressed.
Because this mild solution has three terms, it is convenient to define a sequence of stopping times $\tau_n = \inf\{t>0: |u(t)|_{L^\infty} \geq 3^n\}$ that record the time required for the spatial $L^\infty$ norm of the solutions to triple.

 The first two terms of the right-hand-side of \eqref{eq:intro-mild} are easily controlled using properties of elliptic semigroups and arguments similar to those used in the ODE case. Most of the analysis in this article focuses on the stochastic convolution term. We rely on a moment estimate on the stochastic convolution, which is due to Cerrai \cite{c-2003,c-2009}. This moment estimate is stated below as Lemma \ref{lem:moment} and its proof is included in the appendix of the paper. Lemma \ref{lem:moment} claims that for $\zeta \in (0, 1-\eta)$ and large enough $p>1$, there exists $C=C(p,\zeta)>0$ such that whenever $\Phi(t,x)$ is an adapted process that is almost surely bounded by a constant $M$ and for any $\e>0$,
 \begin{equation} \label{eq:intro-moment}
   \E \sup_{t \in [0,\e]} \left|\int_0^t S(t-s)\Phi(s)dw(s) \right|_{L^\infty}^p \leq CM^p \e^{\frac{p(1-\eta-\zeta)}{2}}.
 \end{equation}
 This moment bound tells us that the stochastic integral is $\gamma$-H\"older continuous in time for any $\gamma \in \left( 0, \frac{1-\eta}{2} \right)$. Heuristically, the growth rate restrictions \eqref{eq:intro-bound} are related to the regularity in time of the Lebesgue and stochastic integrals. The Lipschitz continuity of the Lebesgue integral $t \mapsto \int_0^t S(t-s)f(u(s))ds$ means that $f$ can grow as fast as any $h$ satisfying \eqref{eq:intro-osgood}, while the $\gamma$-H\"older continuity of the stochastic integral means that $\sigma$ can only grow as fast as $|u|^{1-\gamma}(h(|u|))^\gamma$.

Using  moment estimates like \eqref{eq:intro-moment} and the Borel-Cantelli Lemma, we can show that eventually the times $(\tau_n-\tau_{n-1})$ required for the $L^\infty$ norm of $u$  to triple have the property that $(\tau_{n}-\tau_{n-1})\geq a_n$ for a non-random, explicit  sequence such that $\sum a_n =+\infty$.  Specifically, $a_n = \min \left\{\frac{3^{n-1}}{h(3^n)},\frac{1}{n} \right\}$. Therefore $\lim_n \tau_n=+\infty$ with probability one and the mild solutions to \eqref{eq:SPDE} never explode.

This method is both straightforward and also quite adaptable. Similar arguments could be applied to reaction-diffusion equations with spatially homogeneous noise, equations with time-dependent differential operators, stochastic wave equations, or equations with fractional differential operators or fractional stochastic noises.  For simplicity, we focus on the reaction-diffusion equation with noise that is white in time and possibly correlated in space so that we can use BDG and Ito isometry arguments.

The assumptions are outlined in Section \ref{S:assum}. Several examples of admissible superlinear forcing terms are presented in Section \ref{S:examples} and the results are contrasted with those of \cite{dkz-2019}. The tripling times are identified in Section \ref{S:times}. The proof of the main result is in Section \ref{S:proof}. An appendix collects some known results about the moments of stochastic convolutions.

\section{Assumptions and the main result} \label{S:assum}
For $p \geq 1$, let $L^p = L^p(D)$ denote the Banach space of $\phi:D \to \mathbb{R}$ such that the norm
\begin{equation}
  |\phi|_{L^p} : = \left(\int_D |\phi(x)|^pdx \right)^{\frac{1}{p}}< +\infty.
\end{equation}
The $L^\infty$ norm is defined to be the supremum norm
\begin{equation}
  |\phi|_{L^\infty}: = \sup_{x \in D} |\phi(x)|.
\end{equation}
Let $C(\bar D)$ denote the closed subset of $\phi \in L^\infty$ such that $\phi: \bar D \to \mathbb{R}$ is continuous.

We now introduce the main assumptions of this article. The elliptic operator $\mathcal{A}$ in \eqref{eq:SPDE} regularizes the solutions while the stochastic perturbations $\dot{w}$ make the solutions more rough. Assumptions \ref{assum:A} and \ref{assum:noise}, below, were first formulated by Cerrai \cite{c-2003,c-2009}. These assumptions guarantee a kind of balance between the regularization of $\mathcal{A}$ and the roughness of $\dot{w}$ to guarantee that the mild solutions to \eqref{eq:SPDE} are continuous in space and time.  Importantly, Assumptions \ref{assum:A} and \ref{assum:noise} imply Lemma \ref{lem:moment}, which establishes moment bounds on the stochastic convolution terms of the mild solution.

\begin{assumption} \label{assum:A}
  $\mathcal{A}$ is a second-order elliptic operator. Without loss of generality, we assume that $\mathcal{A}$ is self-adjoint (see \cite[Section 3]{c-2003} for discussion of why there is no loss of generality)
  \begin{equation}
    \mathcal{A} g(x) = \sum_{i=1}^d \sum_{j=1}^d \frac{\partial }{\partial x_j} \left(  a_{ij}(x) \frac{\partial g}{\partial x_i}(x) \right).
  \end{equation}
  In the above expression the $a_{ij}$ are twice continuously differentiable. The matrix $(a_{ij})$ is elliptic in the sense that there exists $c>0$ such that for any $x, \xi \in \mathbb{R}^d$,
  \begin{equation}
    \sum_{i=1}^d \sum_{j=1}^d a_{ij}(x) \xi_i \xi_j \geq c \sum_{i=1}^d \xi_i^2.
  \end{equation}

  We assume that $D$ and its boundary are regular enough so that there exists a complete orthonormal basic of $L^2(D)$ made of eigenvalues for $\mathcal{A}$ with the imposed Dirichlet boundary conditions. Specifically, there exist $0 \leq \alpha_k \leq \alpha_{k+1}$ and $e_k \in L^2(D)$ such that
  \begin{equation}
    \mathcal{A} e_k(x) = -\alpha_k e_k.
  \end{equation}
  By elliptic regularity, $e_k \in L^\infty(D)$ \cite[Chapter 6.5]{evans}.
  We also assume that $D$ is regular enough for the fractional Sobolev embedding theorem to be valid \cite{sobolev}.
\end{assumption}
\begin{remark}
  This theory will also work for other kinds of boundary conditions. See \cite{c-2003}
\end{remark}

\begin{assumption} \label{assum:noise}
  Formally the noise is the sum
  \begin{equation}
    dw(t,x) = \sum_{j=1}^\infty \lambda_j e_j(x)d\beta_j(t)
  \end{equation}
  where $\beta_j$ is a sequence of independent one-dimensional Brownian motions defined on a probability space $(\Omega, \mathcal{F}, \Pro)$, and $e_j$ are the eigenfunctions of $\mathcal{A}$. Furthermore, there exist $q \in [2,+\infty]$ and $\theta \in (0,1)$ such that the sequence $\lambda_j\geq0$ satisfies
  \begin{equation}
    \sum_{j=1}^\infty \lambda_j^q |e_j|_{L^\infty}^2 <+\infty, \text{ if } q<+\infty \ \  \ \text{ or } \ \ \
    \sup_j \lambda_j< +\infty \text{ if } q = +\infty,
  \end{equation}
  \begin{equation}
    \sum_{k=1}^\infty \alpha_k^{-\theta} |e_k|_{L^\infty}^2 < +\infty
  \end{equation}
  and
  \begin{equation} \label{eq:cerrai-constant}
    \eta: = \frac{\theta(q-2)}{q}<1, \ \ \ \ \  \left(\eta := \theta <1  \text{  if  } q = +\infty.\right).
  \end{equation}
\end{assumption}

\begin{assumption} \label{assum:f-sigma}
  The reaction term $f:\mathbb{R} \to \mathbb{R}$ and the multiplicative noise term $\sigma:\mathbb{R} \to \mathbb{R}$ have the following properties
  \begin{itemize}
    \item $f$ is locally Lipschitz continuous,
    \item $\sigma$ is locally Lipschitz continuous,
    \item There exists a positive, increasing $h:[0,+\infty) \to [0,+\infty)$  that satisfies
        \begin{equation} \label{eq:h-osgood}
          \int_c^\infty \frac{1}{h(u)}du = +\infty \text{ for all } c>0.
        \end{equation}
         such that
        \begin{equation}
          |f(u)| \leq h(|u|) \text{ for } u \in \mathbb{R}
        \end{equation}
        and there exists $\gamma \in \left( 0, \frac{1-\eta}{2}\right)$ where $\eta$ is defined by \eqref{eq:cerrai-constant} such that
        \begin{equation}
          |\sigma(u)| \leq |u|^{1-\gamma}(h(|u|))^\gamma \text{ for }  |u|>1.
        \end{equation}
  \end{itemize}
\end{assumption}

\begin{remark}
  Notice that if $h$ is superlinear, then $u \mapsto |u|^{1-\gamma}(h(|u|))^\gamma$ is superlinear too.
\end{remark}

\begin{assumption} \label{assum;init-data}
  The initial data $u_0: \bar{D} \to \mathbb{R}$ is uniformly bounded.
\end{assumption}

\begin{definition}
A random field $u(t,x)$ is called a \textit{local} mild solution to \eqref{eq:SPDE} if $u(t,x)$ is a solution to the integral equation
\begin{equation} \label{eq:mild}
  u(t) = S(t)u_0 + \int_0^t S(t-s)f(u(s))ds + \int_0^t S(t-s)\sigma(u(s))dw(s).
\end{equation}
for all $t<T_n := \inf\{t>0: |u(t)|_{L^\infty} \geq n\}$ for all $n$.
In the above expression, the spatial variable $x$ has been suppressed (see \cite{dpz}). $S(t)$ is the $C_0$ semigroup generated by $\mathcal{A}$ with the imposed boundary conditions.
\end{definition}

\begin{definition}
  A mild solution is a  \textit{global solution} if \eqref{eq:mild} holds for all $t\geq0$ with probability one. Equivalently, a mild solution is a  \textit{global solution} if $\lim_{ n \to \infty} T_n=+\infty$ with probability one.
\end{definition}

The main result of this article is Theorem \ref{thm:main}.

\begin{theorem} \label{thm:main}
  Under Assumptions \ref{assum:A}--\ref{assum;init-data}, the there exists a unique mild solution $u$ solving \eqref{eq:mild} and the mild solution is global.
\end{theorem}
\section{Examples and comparison with \cite{dkz-2019}} \label{S:examples}
In \cite{dkz-2019}, the authors consider the case of a one-dimensional heat equation driven by space-time white noise.
\begin{equation}
  \begin{cases}
    \frac{\partial u}{\partial t}(t,x) = \frac{\partial^2 u}{\partial x^2} (t,x) + f(u(t,x)) + \sigma(u(t,x))\dot{w}(t,x)\\
    u(t,0) = u(t,\pi) = 0.
  \end{cases}
\end{equation}
The spatial domain is $D=[0,1]$. The eigenvalues of the elliptic operator $\mathcal{A} = \frac{\partial^2}{\partial x^2}$ are $e_k(x) = \sqrt{2} \sin(k\pi x)$ and the eigenvalues are $-\alpha_k = -\pi^2 k^2$.

In the language of our Assumption \ref{assum:noise}, white noise means that $\lambda_j\equiv 1$. Consequently $q = +\infty$ and $\theta \in (\frac{1}{2},1)$ can be chosen arbitrarily close to $\frac{1}{2}$ so that $\sum k^{-2\theta}$ converges. This implies that $\eta\in \left(\frac{1}{2},1\right)$. The exponent $\gamma$ in Assumption \ref{assum:f-sigma} can be arbitrarily close to $\frac{1-\eta}{2}$ so $\gamma < \frac{1}{4}$. In \cite{dkz-2019}, they assume that $|f(u)|\leq C(1 + |u|\log|u|)$ for some $C>0$ and that $\sigma \in o(|u|(\log|u|)^{\frac{1}{4}})$.

First, we demonstrate that Assumption \ref{assum:f-sigma} holds for a larger class of reaction terms $f$ than the assumptions of \cite{dkz-2019}. First, observe that $h_1(u) = C (1 + |u|\log|u|)$ satisfies Assumption \ref{assum:f-sigma} because $\int \frac{1}{h_1(u)}du = +\infty$. As Example \ref{ex:1} demonstrates below, Assumption \ref{assum:f-sigma} allows for $f$ terms that grow much faster than $|u|\log|u|$. 

\begin{example} \label{ex:1}
  Define
  \begin{align}
    &L_1(u) = \log(u)\\
    &L_{n+1}(u) = \log(L_n(u))\\
    &E_1 = \exp(1)\\
    &E_{n+1} = \exp(E_n).
  \end{align}
  Let $h_n(u) = (u + E_n)\prod_{k=1}^n L_k(u + E_n)$. The constants $E_n$ are added to guarantee that $L_k(u+E_n)\geq 1$. Direct calculations prove that $h_n$ is positive and increasing and that for any $c_1<c_2$,
  \begin{equation}
    \int_{c_1}^{c_2} \frac{1}{h_n(u)}du = L_{n+1}(c_2 + E_n) - L_{n+1}(c_1 + E_n).
  \end{equation}
  By letting $c_2 \uparrow +\infty$ we see that each $h_n$ satisfies \eqref{eq:h-osgood}. The reaction term $f(u) = h_n(|u|)$ satisfies our Assumption \ref{assum:f-sigma}, but grows much faster than $|u|\log|u|$.
\end{example}

Now we compare the Assumption \ref{assum:f-sigma} restrictions on $\sigma$ with those of \cite{dkz-2019}.
If $h_n$ is one of the examples from Example \ref{ex:1} and $\gamma \in (0,1/4)$, then any $\sigma$ satisfying
\begin{equation}
  |\sigma(u)| \leq |u|^{1-\gamma}(h_n(u))^\gamma
\end{equation}
has the property that $\sigma \in o(|u|(\log|u|)^{\frac{1}{4}})$ because
\begin{equation}
  \lim_{|u| \to \infty} \frac{(h_n(|u|))^\gamma}{|u|^{\gamma}(\log|u|)^{\frac{1}{4}}}=0.
\end{equation}

Example \ref{ex:2}, below, will show that there exists a $\sigma \in o(|u|(\log|u|)^{\frac{1}{4}})$ that fails to satisfy Assumption \ref{assum:f-sigma}. After that, Example \ref{ex:3} will prove that there exist $\sigma$ that satisfy Assumption \ref{assum:f-sigma} but $\sigma \not \in o(|u|(\log|u|)^{\frac{1}{4}})$. These two examples demonstrate that the restrictions on $\sigma$ identified in this paper are neither fully weaker or stronger than those in \cite{dkz-2019} in the case of a one-dimensional heat equation exposed to space-time white noise. The results of this article also allow for other types of elliptic operators and noise terms in arbitrary spatial dimensions.

\begin{example} \label{ex:2}
  $\sigma(u) = (e^e+  u) \frac{(\log(e^e + u))^{\frac{1}{4}}}{\log\log(e^e + u)}$ satisfies the assumptions of \cite{dkz-2019} because $\sigma \in o(u(\log(u))^\frac{1}{4})$, but violates the assumptions of this paper. The constant $\eta$ defined in \eqref{eq:cerrai-constant} can be any value $\eta \in (1/2,1)$ implying that $\frac{1-\eta}{2} \in (0, 1/4)$. If $h:[0,+\infty) \to [0,+\infty)$ is such that
  \begin{equation}
    |\sigma(u)| \leq    |u|^{1-\gamma}(h(|u|))^\gamma \text{ for } |u|>1
  \end{equation}
  for some $\gamma \in (0,1/4)$, then
  \begin{equation}
    h(u) \geq c(e^e + u) \frac{(\log(e^e + u))^{\frac{1}{4\gamma}}}{(\log\log(e^e + u))^{\frac{1}{\gamma}}}.
  \end{equation}
  Choose any $\delta \in \left( 1, \frac{1}{4\gamma} \right)$, then for large $u$,
  \begin{equation}
    h(u) > c (e^e + u) (\log(e^e + u))^\delta.
  \end{equation}
  Direct calculations show that any such $h$ has the property that for large $c>0$,
  \begin{equation}
    \int_c^\infty \frac{1}{h(u)}du < +\infty,
  \end{equation}
  which violates Assumption \ref{assum:f-sigma}.
\end{example}

On the other hand, the assumptions of this paper allow for certain $\sigma$ terms that violate the assumptions of  \cite{dkz-2019}.

\begin{example} \label{ex:3}
  Let $g: [0,+\infty) \to [0,+\infty)$ be an arbitrary increasing function that increases to $+\infty$. It is helpful to think of $g(u)=u^2$. Even though $g$ may have the property that
  \begin{equation}
    \int_1^\infty \frac{1}{g(u)}du < +\infty,
  \end{equation}
  it is always possible to construct a positive, increasing  $h: [0,+\infty) \to [0,+\infty)$ such that
  \begin{equation}
    \limsup_{u \to +\infty} \frac{h(u)}{g(u)} =1,
  \end{equation}
  but
  \begin{equation}
    \int_{c}^\infty \frac{1}{h(u)}du = +\infty.
  \end{equation}
  In other words, we can construct an $h$ and a sequence of $u_n \uparrow +\infty$ such that $h$ is increasing, $h(u_n) = g(u_n)$ for all $n \in \mathbb{N}$, but $h$ satisfies \eqref{eq:h-osgood} anyway. If $g$ is increasing and convex then the resulting $h$ is strictly increasing, but it cannot be convex.

  Define $u_1=1$ and $u_{n+1} = u_n + 1 + g(u_n)$. Define $h(x)$ so that $h(u_n) = g(u_n)$ for all $n \in \mathbb{N}$. For $u \in [u_n,u_{n+1}]$, define $h$ so that $\frac{1}{h(u)}$ is linear. The integral
  \begin{equation}
    \int_{u_n}^{u_{n+1}} \frac{1}{h(u)}du = \frac{1}{2} \left( u_{n+1} - u_n \right) \left(\frac{1}{g(u_n)} + \frac{1}{g(u_{n+1})} \right)
  \end{equation}
  is the area of a trapezoid. By the construction of $u_n$, the difference $u_{n+1} - u_n \geq g(u_n)$, proving that
  \begin{equation}
    \int_{u_n}^{u_{n+1}} \frac{1}{h(u)}du \geq \frac{1}{2} \text{ for any } n \in \mathbb{N}.
  \end{equation}

  Therefore, we have proved that
  \begin{equation}
  \int_{u_1}^\infty \frac{1}{h(u)}du = +\infty \text{ and } h(u_n) = g(u_n).
  \end{equation}

  Now let $\gamma \in \left( 0, \frac{1}{4} \right)$ and define
  \begin{equation}
    \sigma(u) = 1  + |u|^{ 1-\gamma} (h(|u|))^{\gamma}
  \end{equation}
  Such a $\sigma$ satisfies Assumption \ref{assum:f-sigma}. Consider the example where $g(u) = u^2$. In this case, there is a sequence of $u_n \uparrow +\infty$ such that $h(u_n) = u_n^2$ and $\sigma(u_n) = 1 + u_n^{1 + \gamma}$. This means that this $\sigma \not \in o(|u|(\log|u|)^{1/4})$.
\end{example}

\section{Tripling times -- local existence and uniqueness} \label{S:times}
The main idea of the proof is to define the tripling times
\begin{align}
  &\tau_n = \inf\{t>0: |u(t)|_{L^\infty}\geq 3^n\}.
\end{align}
Using these tripling times, we can prove local existence and uniqueness for the stochastic-reaction diffusion equations using standard arguments. Tripling times are convenient because the mild formulation \eqref{eq:mild} has three terms.

\begin{definition}
  Define the cutoff functions
  \begin{equation}
    f_n(u)
    = \begin{cases}
      f(-3^n) & \text{ if } u< -3^n\\
      f(u) & \text{ if } u \in [-3^n, 3^n]\\
      f(3^n) & \text{ if } u> 3^n
    \end{cases}
  \end{equation}
  and
  \begin{equation} \label{eq:sigman}
    \sigma_n(u)
    = \begin{cases}
      \sigma(-3^n) & \text{ if } u< -3^n\\
      \sigma(u) & \text{ if } u \in [-3^n, 3^n]\\
      \sigma(3^n) & \text{ if } u> 3^n
    \end{cases}
  \end{equation}
  Because $f$ and $\sigma$ are assumed to be locally Lipschitz continuous, $f_n$ and $\sigma_n$ are globally Lipschitz continuous.
\end{definition}
\begin{definition}
  Define $u_n$ to be the mild solution to the reaction-diffusion equation
  \begin{equation} \label{eq:u_n-mild}
    u_n(t) = S(t) u(0) + \int_0^t S(t-s)f_n(u_n(s))ds + \int_0^t S(t-s)\sigma_n(u_n(s))dw(s).
  \end{equation}
\end{definition}

By the results of \cite{c-2003,c-2009}, for each $n$, there exists a unique mild solution $u_n$. The local mild solution \eqref{eq:mild} can be uniquely defined as
\begin{equation} \label{eq:consistency}
  u(t) = u_n(t) \text{ for all } t < \tau_n.
\end{equation}

The main result of this paper is that $u(t)$ is a global solution.
We will prove that the solutions never explode by showing that
\begin{equation}
  \Pro \left(\sup_n \tau_n = +\infty \right) = 1.
\end{equation}
To this end, we define the real sequence
\begin{equation} \label{eq:an}
  a_n:= \min\left\{\frac{3^{n-1}}{h(3^{n})}, \frac{1}{n} \right\}
\end{equation}
where $h$ is the increasing function in Assumption \ref{assum:f-sigma}.
In the proof of the main result, we will demonstrate via the Borel-Cantelli Lemma that $\tau_n -\tau_{n-1}< a_n$ only a finite number of times.
Now we establish some important properties of $a_n$.
\begin{lemma}
  The real-valued sequence $a_n$ satisfies
  \begin{equation}
    \sum_{n=1}^\infty a_n = +\infty
  \end{equation}
\end{lemma}
\begin{proof}
  First, we observe that
  \begin{equation}
    3^{n+1} - 3^{n} = (3-1) 3^{n} =(2) (3^{n}) = 6(3^{n-1}).
  \end{equation}
  Therefore, we can bound
  \begin{align*}
    \frac{3^{n-1}}{h(3^n)}
    \geq \frac{3^{n+1} - 3^n}{6h(3^n)}
    \geq \frac{1}{6} \int_{3^n}^{3^{n+1}} \frac{1}{h(x)}dx.
  \end{align*}
  The last equality holds because of the assumption that $h$ is increasing.
  Therefore,
  \begin{equation}
    \sum_{n=1}^\infty \frac{3^{n-1}}{h(3^n)} \geq \frac{1}{6} \int_{3}^\infty \frac{1}{h(x)}dx = +\infty.
  \end{equation}

  Cauchy's condensation theorem guarantees that
   \begin{equation}
     \sum_{n=1}^\infty a_n = \sum_{n=1}^\infty \min\left\{\frac{3^{n-1}}{h(3^{n})}, \frac{1}{n} \right\} = +\infty.
   \end{equation}
\end{proof}

\section{Proof of the main result} \label{S:proof}
Notice that for $n \in \mathbb{N}$ and $t >0$,
  \begin{align}
    u(\tau_{n-1} + t) = &S(t) u(\tau_{n-1}) + \int_{\tau_{n-1}}^{t + \tau_{n-1}} S(t + \tau_{n-1}-s) f(u( s))ds \nonumber\\
    &+ \int_{\tau_{n-1}}^{t + \tau_{n-1}} S(t+\tau_{n-1} -s) \sigma(u(s))dw(s) \nonumber\\
    =:& S(t) u (\tau_{n-1})  + I_n(t) + Z_n(t)
  \end{align}

  A main idea of the proof is that
  \begin{align*}
    \Pro(\tau_n - \tau_{n-1} < a_n) = \Pro \left(\sup_{t \in [0, (\tau_n - \tau_{n-1}) \wedge a_n] } |u(t + \tau_{n-1})|_{L^\infty} \geq 3^n \right).
  \end{align*}

  First we show that $|S(t)u(\tau_{n-1})|_{L^\infty}$ and $|I_n(t)|_{L^\infty}$ are uniformly bounded for $t \in [0,(\tau_n - \tau_{n-1})\wedge a_n]$. Then the proof simplifies to analysis of the $Z_n(t)$ terms.

\begin{lemma} \label{lem:contr}
  For any $n \in \mathbb{N}$ and $t>0$,
  \begin{equation}
    |S(t) u(\tau_{n-1})|_{L^\infty} \leq 3^{n-1}
  \end{equation}
  with probability one.
\end{lemma}
\begin{proof}
  $S(t)$ is a contraction in the supremum norm . Therefore,
  \begin{equation}
    |S(t) u(\tau_{n-1})|_{L^\infty} \leq |u(\tau_{n-1})|_{L^\infty} \leq 3^{n-1}.
  \end{equation}
\end{proof}

\begin{lemma} \label{lem:int-bound}
For any $n >1$,
  \begin{equation}
    \sup_{t \in [0,(\tau_n - \tau_{n-1}) \wedge a_n]} \left|I_n(t) \right|_{L^\infty}
    \leq 3^{n-1}
  \end{equation}
  with probability one.
\end{lemma}

\begin{proof}
  For any $t \in [0, (\tau_{n} - \tau_{n-1}) \wedge a_n]$,
  \begin{equation}
    \sup_{x \in D} |f(u(\tau_{n-1} + t,x))| \leq \sup_{x \in D} h(|u(t + \tau_{n-1},x)|) \leq h(3^n)
  \end{equation}
  because $h$ is increasing and $|u(t+\tau_{n-1})|\leq 3^n$ if $t \leq \tau_n - \tau_{n-1}$.
  $S(t)$ is a contraction semigroup in the space of continuous functions. Therefore,
  \begin{align} \label{eq:Lp-stoch-int}
    &\sup_{t \in [0,(\tau_n - \tau_{n-1}) \wedge a_n]} \left|\int_{\tau_{n-1}}^{\tau_{n-1}+t} S(\tau_{n-1} + t-s) f(u(s))ds \right|_{L^\infty}\nonumber\\
    &\leq \sup_{t \in [0,(\tau_n - \tau_{n-1}) \wedge a_n]} \int_{\tau_{n-1}}^{\tau_{n-1}+t} \left| f(u(s))\right|_{L^\infty}ds \nonumber\\
    &\leq a_n h(3^{n}) \leq 3^{n-1}.
  \end{align}
  The last inequality is from the definition of $a_n$ \eqref{eq:an}.
\end{proof}

The next moment estimate is due to Cerrai \cite{c-2003,c-2009}. The proof is in the appendix of this article.
\begin{lemma} \label{lem:moment}
  For any  $\zeta \in (0,1-\eta)$, $p> \max\left\{\frac{2}{1-\eta-\zeta}, \frac{d}{\zeta}\right\}$, where $d$ is the spatial dimension, there exists $C = C(\zeta,p)>0$ such that for any $\e>0, M>0$ and any adapted process $\Phi :[0,T]\times D \times \Omega \to \mathbb{R}$ such that $\Pro(\sup_{s \in [0,\e]}|\Phi(s)|_{L^\infty} \leq M)=1$,
  \begin{equation}
    \E \sup_{t \in [0,\e]}\left|\int_0^t S(t-s)\Phi(s)dw(s)\right|_{L^\infty}^p \leq C M^p \e^{\frac{p(1-\eta-\zeta)}{2}}.
  \end{equation}
\end{lemma}

\begin{lemma} \label{lem:prob-tau<a}
There exist $C>0$ and $r>1$ such that for all $n \in \mathbb{N}$,
  \begin{equation}
    \Pro(\tau_n - \tau_{n-1} < a_n) \leq C n^{-r}.
  \end{equation}
\end{lemma}

\begin{proof}
  Recall that
  \begin{align}
    u(t + \tau_{n-1}) = &S(t) u(\tau_{n-1}) + \int_{\tau_{n-1}}^{t + \tau_{n-1}} S(t + \tau_{n-1}-s) f(u( s))ds \nonumber\\
    &+ \int_{\tau_{n-1}}^{t + \tau_{n-1}} S(\tau_{n-1} + t-s) \sigma(u(s))dw(s) \nonumber\\
    =:& S(t) u (\tau_{n-1})  + I_n(t) + Z_n(t)
  \end{align}
  By Lemma \ref{lem:contr},  $|S(t) u(\tau_{n-1})|_{L^\infty} \leq 3^{n-1}$ for all $t >0$. By Lemma \ref{lem:int-bound},  $|I_n(t)| \leq 3^{n-1}$ for $t \in [0,(\tau_n -\tau_{n-1})\wedge a_n]$. Therefore, the only way that $|u(\tau_{n-1} + t)|_{L^\infty}\geq 3^n$ for some $t \in [0,(\tau_n-\tau_{n-1}) \wedge a_n]$, is that that the stochastic integral must satisfy $\sup_{t \in [0,(\tau_n - \tau_{n-1})\wedge a_n]}|Z_n(t)| > 3^{n-1}$.

  Therefore,
  \begin{align*}
    \Pro(\tau_n - \tau_{n-1}< a_n)
    \leq \Pro\left( \sup_{t \in [0,a_n \wedge (\tau_n - \tau_{n-1})]} \left| Z_n(t) \right|_{L^\infty} \geq 3^{n-1} \right)
  \end{align*}
  Now define
  \begin{equation}
    \tilde{Z}_n(t) = \int_{\tau_{n-1}}^{t + \tau_{n-1}} S(\tau_{n-1} + t-s) \sigma_n(u_n(s))dw(s)
  \end{equation}
  where $\sigma_n$ and $u_n$ are defined in \eqref{eq:sigman} and \eqref{eq:u_n-mild}.
  Notice that $\tilde{Z}_n(t) = Z_n(t)$ for $ t \in [0,\tau_n-\tau_{n-1}]$.
  Therefore,
  \begin{align*}
    &\Pro(\tau_n -\tau_{n-1}< a_n)
    \leq \Pro\left( \sup_{t \in [0,a_n \wedge (\tau_n - \tau_{n-1})]} \left| Z_n(t) \right|_{L^\infty} \geq 3^{n-1} \right)\\
    &\leq  \Pro\left( \sup_{t \in [0,a_n \wedge (\tau_n - \tau_{n-1})]} \left| \tilde Z_n(t) \right|_{L^\infty} \geq 3^{n-1} \right)\\
    &\leq \Pro\left( \sup_{t \in [0,a_n ]} \left| \tilde Z_n(t) \right|_{L^\infty} \geq 3^{n-1} \right)
  \end{align*}
  By Chebyshev's inequality, for $p>1$ large enough,
  \begin{align*}
    &\Pro(\tau_n -\tau_{n-1}< a_n) \leq \Pro\left( \sup_{t \in [0,a_n ]} \left| \tilde Z_n(t) \right|_{L^\infty} \geq 3^{n-1} \right)\\
    &\leq 3^{-p(n-1)}\E \sup_{t \in [0,a_n]} \left| \int_{\tau_{n-1}}^{\tau_{n-1} +t} S(t+\tau_{n-1}-s)\sigma_n(u_n(s))dw(s) \right|^p.\\
  \end{align*}
   Given $\gamma \in \left( 0, \frac{1-\eta}{2} \right)$ from Assumption \ref{assum:f-sigma}, choose $\zeta \in \left(0, 1-\eta - 2\gamma\right)$, and $p> \max\{\frac{2}{1-\eta-2\gamma-\zeta}, \frac{d}{\zeta}\}$.
  By Lemma \ref{lem:moment}  and the fact that $|\sigma_n(u_n(s))|_{L^\infty} \leq  3^{n(1-\gamma)}(h(3^n))^{\gamma}$ for all $s \in [\tau_{n-1}, \tau_n]$,
  \begin{align}
     &\Pro(\tau_n - \tau_{n-1} < a_n) \nonumber\\
     &\leq C 3^{-p(n-1)} (a_n)^{p\left(\frac{1-\eta-\zeta}{2}\right)} \left( 3^{n (1-\gamma) p} (h(3^n))^{p\gamma} \right)\nonumber\\
     &\leq C (a_n)^{p\left( \frac{1-\eta-\zeta}{2}  - \gamma\right)}.
  \end{align}
  In the last step we used the definition that $a_n \leq \frac{3^{n-1}}{h(3^{n})}$ to see that $\left(\frac{h(3^n)}{3^n}\right)^{p\gamma} \leq 3^{-p \gamma} a_n^{-p \gamma}$.
  The constant $C$ depends only on the choices of $p$ and $\zeta$ and it does not depend on $n$. Therefore, because $a_n$ is defined to be less than $\frac{1}{n}$,
  \begin{equation}
    \Pro(\tau_n - \tau_{n-1} < a_n) \leq C n^{-\left(\frac{1-\eta-\zeta }{2}  -\gamma\right)p}
  \end{equation}
  where the exponent was chosen so that $r:=\left(\frac{1-\eta-\zeta}{2}-\gamma\right)p>1$.
\end{proof}

\begin{proof}[Proof of Theorem \ref{thm:main}]
  By Lemma \ref{lem:prob-tau<a}, There exist $C>0$ and $r>1$ such that
  \[\Pro(\tau_n - \tau_{n-1}<a_n) \leq C n^{-r}.\]
  Because $r>1$,
  \begin{equation}
    \sum_{n=1}^\infty \Pro(\tau_n - \tau_{n-1} < a_n) < +\infty.
  \end{equation}
  By the Borel-Cantelli Lemma, with probability one there exists $N(\omega)$ such that for all $n \geq N(\omega)$, $\tau_n(\omega) -\tau_{n-1}(\omega)\geq a_n$. Then because $\sum a_n =+\infty$, it follows that
  \begin{equation}
    \Pro(\sup_n \tau_n = +\infty) =  \Pro \left( \sum_{n=2}^\infty (\tau_n- \tau_{n-1}) = +\infty \right)=1.
  \end{equation}
  This proves that mild solutions to \eqref{eq:SPDE} are global in time.
\end{proof}

\begin{appendix}
  \section{Moments of the stochastic integral}
    Let $S(t)$ be the semigroup generated by $\mathcal{A}$ with the imposed boundary conditions.
    Let $\Phi:[0,T]\times D \times \Omega \to \mathbb{R}$ be an adapted process that is almost surely bounded in the sense that there exists $M>0$ such that
    \begin{equation}
      \Pro \left(\sup_{t \in [0,T]} \sup_{x \in D} |\Phi(t,x)|\leq M \right)=1.
    \end{equation}

    In this appendix, we prove the moment estimates of the stochastic convolution
    \begin{equation}
      Z(t) = \int_0^t S(t-s)\Phi(s)dw(s)
    \end{equation}
    stated in Lemma \ref{lem:moment}.

  \begin{lemma}[Factorization (Theorem 5.10 of \cite{dpz})] \label{lem:factorization}
    For any $\alpha \in \left( 0, \frac{1-\eta}{2}\right)$ where $\eta$ is given in \eqref{eq:cerrai-constant},
    \begin{equation}
      Z(t) = \frac{\sin(\alpha \pi)}{\pi} \int_0^t (t-s)^{\alpha-1}S(t-s) Z_\alpha(s)ds
    \end{equation}
    where
    \begin{equation}
      Z_\alpha(t) = \int_0^t (t-s)^{-\alpha} S(t-s) \Phi(s)dw(s).
    \end{equation}
  \end{lemma}

  \begin{lemma} \label{lem:sum-of-squares}
    For any $\varphi \in L^\infty$, $t>0$ and $x \in D$,
    \begin{equation}
      \sum_{j=1}^\infty \left((S(t) \phi e_j)(x) \right)^2 \leq C|\phi|_{L^\infty}^2 t^{-\theta}
    \end{equation}
    where $\theta>0$ is defined in Assumption \ref{assum:noise}.
  \end{lemma}
  \begin{proof}
    The semigroup has a kernel representation
    \begin{equation}
      (S(t)\phi)(x) = \int_D K(t,x,y) \phi(y)dy
    \end{equation}
    where
    \begin{equation}
      K(t,x,y) = \sum_{k=1}^\infty e^{-\alpha_k t} e_k(x) e_k(y).
    \end{equation}
    Therefore,
    \begin{align*}
      \sum_{j=1}^\infty \left(S(t) \phi e_j \right)^2
      = \sum_{j=1}^\infty \left(\int_D K(t,x,y)\phi(y)e_j(y) \right)^2.
    \end{align*}
    Because the $\{e_j\}$ form a complete orthonormal basis of $L^2(D)$,
    the above expression equals
    \begin{equation*}
      = \int_D (K(t,x,y))^2 |\phi(y)|^2dy
      \leq |\phi|_{L^\infty}^2 \int_D (K(t,x,y))^2 dy.
    \end{equation*}
    Now we estimate $\int_D (K(t,x,y))^2 dy$. Because the $\{e_k\}$ form a complete orthonormal basis,
    \begin{align*}
      \int_D (K(t,x,y))^2 dy
      = \sum_{k=1}^\infty e^{-2\alpha_k t} |e_k(x)|^2
      \leq  \sum_{k=1}^\infty e^{-2\alpha_k t} |e_k|_{L^\infty}^2.
    \end{align*}
    Finally, because the constant $c_\theta:= \sup_{x>0} x^\theta e^{-x}<+\infty$,
    \begin{equation*}
      \sum_{k=1}^\infty e^{-2\alpha t} |e_k|_{L^\infty}^2\leq c_\theta (2t)^{-\theta} \sum_{k=1}^\infty \alpha_k^{-\theta} |e_k|_{L^\infty}^2,
    \end{equation*}
    which converges by Assumption \ref{assum:noise}.
  \end{proof}

  \begin{lemma} \label{lem:weighted-squares}
    For any $\varphi \in L^\infty$, $t>0$ and $x \in D$,
    \begin{equation}
      \sum_{j=1}^\infty \lambda_j^2 \left((S(t) \phi e_j)(x) \right)^2 \leq C|\phi|_{L^\infty}^2 t^{-\eta}
    \end{equation}
    where $\eta = \frac{\theta(q-2)}{q}$ is defined in \eqref{eq:cerrai-constant}
  \end{lemma}

  \begin{proof}
    By the H\"older inequality with exponents $\frac{q}{2}$ and $\frac{q}{q-2}$,
    \begin{align*}
      &= \sum_{j=1}^\infty \lambda_j^2 \left((S(t) \phi e_j)(x) \right)^2\\
      &\leq  \left(\sum_{j=1}^\infty \lambda_j^q \left((S(t) \phi e_j)(x) \right)^2
  \right)^{\frac{2}{q}} \left(\sum_{j=1}^\infty \left((S(t) \phi e_j)(x) \right)^2
 \right)^{\frac{q-2}{q}}.
    \end{align*}
    Because $S(t)$ is a contraction semigroup, we bound
    \[|(S(t)\phi e_j)(x)| \leq |\phi|_{L^\infty}|e_j|_{L^\infty}\]
    in the first term of the above product.
    We apply Lemma \ref{lem:sum-of-squares} to bound the second term in the product. Consequently,
    \begin{equation}
      \sum_{j=1}^\infty \lambda_j^2 \left((S(t) \phi e_j)(x) \right)^2
      \leq |\phi|_{L^\infty}^2  C t^{-\frac{\theta(q-2)}{q}} \leq  C|\phi|_{L^\infty}^2 t^{-\eta}.
    \end{equation}
  \end{proof}

  \begin{lemma} \label{lem:Lp-bound-fixed-x}
    For any $\alpha \in (0, \frac{1-\eta}{2})$ and $p\geq 2$, there exists $C = C(p,\alpha)>0$ such that for any adapted $\Phi(t,x,\omega)$ satisfying $\Pro(\sup_{x \in D} \sup_{t \in [0,T]} |\Phi(t,x)|\leq M) = 1$, and any fixed $t>0$, $x \in D$,
    \begin{equation}
      \E \left|Z_\alpha(t,x) \right|^p \leq C M^p t^{\frac{p}{2}(1 - 2\alpha - \eta)}
    \end{equation}

  \end{lemma}

  \begin{proof}
    By the definition of the noise in Assumption \ref{assum:noise},
    for $t>0$ and $x \in D$,
    \begin{equation}
      Z_\alpha(t,x) = \sum_{j=1}^\infty \int_0^t (t-s)^{-\alpha}\left(S(t-s)\Phi(s) e_j \right)(x) \lambda_j d\beta_j(s).
    \end{equation}
    By the BDG inequality, for $p\geq 2$,
    \begin{equation}
      \E \left| Z_\alpha(t,x)\right|^p \leq C \E \left( \int_0^t (t-s)^{-2\alpha}\sum_{j=1}^\infty \lambda_j^2 \left(\left(S(t-s)\Phi(s) e_j \right)(\xi)\right)^2ds \right)^{\frac{p}{2}}
    \end{equation}

    By Lemma \ref{lem:weighted-squares},
    \begin{equation}
      \E \left| Z_\alpha(t,x)\right|^p \leq C \E \left(\int_0^t (t-s)^{-2\alpha -\eta} |\Phi(s)|_{L^\infty}^2ds \right)^{\frac{p}{2}}.
    \end{equation}
    Because $\Phi$ was assumed to be uniformly bounded by $M$, we arrive at the result.
  \end{proof}

  \begin{proof}[Proof of Lemma \ref{lem:moment}]
    An immediate consequence of Lemma \ref{lem:Lp-bound-fixed-x} and the fact that $D$ is bounded is that
    \begin{equation} \label{eq:Lp-bound}
      \E|Z_\alpha(t)|_{L^p(D)}^p = \int_D \E |Z_\alpha(t,x)|^pdx \leq C M^p t^{\frac{p}{2}(1 - 2\alpha - \eta)}.
    \end{equation}
    For $\zeta \in (0,1)$ let $W^{\zeta,p}(D)$ be the fractional Sobolev space with norm
    \begin{equation} \label{eq:Sobolev-norm}
      |\varphi|_{W^{\zeta,p}(D)}^p := \int_D |\varphi(x)|^pdx + \int_D \int_D \frac{|\varphi(x) - \varphi(y)|^p}{|x-y|^{d + \zeta p}}dxdy.
    \end{equation}
    In the above expression, $d$ is the spatial dimension.
    By the Sobolev embedding theorem \cite[Theorem 8.2]{sobolev}, when $\zeta p>d$, $W^{\zeta,p}(D)$ embeds continuously into the space of continuous functions endowed with the supremum norm.
    Furthermore, the semigroup has the property that for $\zeta \in (0,1)$ and $p\geq 1$, there exists $C=C(s,p)$ such that
    \begin{equation} \label{eq:regularize}
      |S(t)\varphi|_{W^{\zeta,p}(D)} \leq C t^{-\frac{\zeta}{2}}|\varphi|_{L^p}.
    \end{equation}
    See, for example \cite[(2.4)]{c-2003}.
    Now choose $\alpha \in \left(0, \frac{1-\eta}{2}\right)$, $\zeta \in (0,\alpha)$ and $p > \max\left\{\frac{d}{\zeta}, \frac{1}{\alpha-\frac{\zeta}{2}}\right\}$.
    Therefore, using Lemma \ref{lem:factorization} and \eqref{eq:regularize}
    \begin{align}
      &|Z(t)|_{L^\infty}^p \leq C\left(\int_0^t (t-s)^{\alpha-1} |S(t-s) Z_\alpha(s)|_{W^{\zeta,p}(D)}ds\right)^p\nonumber\\
      &\leq C \left(\int_0^t (t-s)^{\alpha - 1 - \frac{\zeta}{2}} |Z_\alpha(s)|_{L^p(D)}ds\right)^p.
    \end{align}
    By our choices of $\alpha, \zeta,$ and $p$, and  H\"older inequality
    \begin{equation}
      \E \sup_{t \in [0,T]} |Z(t)|_{L^\infty}^p \leq C T^{p(\alpha - \frac{\zeta}{2}) -1}\E \int_0^T |Z_\alpha(s)|_{L^p}^pds
    \end{equation}
    By \eqref{eq:Lp-bound},
    \begin{equation}
      \E \sup_{t \in [0,T]} |Z(t)|_{L^\infty}^p \leq C M^p T^{ \left(\frac{p(1-\eta - \zeta)}{2} \right)}
    \end{equation}
  \end{proof}

\end{appendix}

\bibliographystyle{amsplain}
\bibliography{super-linear-bounded-domain}
\end{document}